\documentclass[12pt]{amsart}

\usepackage{tikz-cd}
\usetikzlibrary{arrows}
\tikzset{
   commutative diagrams/.cd,
   arrow style=tikz,
   diagrams={>=latex}}

\usepackage{amsmath,amssymb}
\usepackage{setspace}

\newtheorem{theorem}{Theorem}[section]
\newtheorem{lemma}[theorem]{Lemma}
\newtheorem{sublemma}[theorem]{Claim}
\newtheorem{corollary}[theorem]{Corollary}
\newtheorem{proposition}[theorem]{Proposition}

\theoremstyle{definition}

\newtheorem{remark}[theorem]{Remark}

\date{}
\title{On genera of coverings of Torus Bundles}
\author{V\'\i ctor N\'u\~nez}
\email{victor@cimat.mx}
\address{CIMAT, A.P. 402, Guanajuato 36000, Gto., M\'EXICO}

\author[Enrique Ram\'\i rez]{Enrique Ram\'\i rez-Losada}
\email{kikis@cimat.mx}
\address{CIMAT, A.P. 402, Guanajuato 36000, Gto., M\'EXICO}

\author[Jair Remigio]{Jair Remigio-Ju\'arez}
\email{jair.remigio@ujat.mx}
\address{DACB, Universidad Ju\'arez Aut\'onoma de Tabasco, Km.~1 Carr. Cunduac\'an-Jalpa de M\'endez, Cunduac\'an 86690, Tab., M\'EXICO}

\subjclass{Primary 57M12}
\keywords{Heegaard genus, covering space, torus bundle.}

\begin{document}

\maketitle

\begin{abstract}

After showing that a covering space of surface bundles over~$S^1$
factors as a `covering of fibers' followed by a `power covering', we
prove that, for torus bundles, power coverings do not lower Heegaard
genus, and that fiber coverings lower the
genus only in special cases.

\end{abstract}
\section{Introduction.}
Any closed connected 3-manifold $M$ is the union of two handlebodies with
pairwise disjoint interiors. The minimal genus of the handlebodies
among all such decompositions is called the \emph{Heegaard genus} of
$M$, or simply the \emph{genus} of $M$ and is denoted by $g(M)$.
The \emph{rank} of a group is the cardinality of a
minimal set of generators for the group.
 The
rank of the fundamental group of $M$ is called \emph{the rank} of $M$, and
gives a lower bound~$rank(\pi_1(M))\leq g(M)$. For a covering space of
3-manifolds~$\varphi:\tilde M\rightarrow M$, we say that $\varphi$
\emph{lowers the genus} if $g(\tilde M)<g(M)$.

There are two famous questions. First, for a 3-manifold $M$, is the
genus of $M$ equal to the rank of $M$? And, secondly, if
the fundamental group of $M$ contains a finite index subgroup of a
given rank, can the rank of the subgroup be smaller than the rank of $M$? In terms of
covering spaces, we can pose the second question as: Is there a
finite-sheeted covering space of $M$ that lowers the genus? and, 
how large is the lowering? 

P. Shalen states two conjectures (\cite{chalen}):
\begin{enumerate}
\item \label{conj1} For closed connected orientable hyperbolic 3-manifolds rank
  equals genus.

\item \label{conj2} For closed connected orientable hyperbolic 3-manifolds a
  finite-sheeted covering space lowers the genus at most by one.

\end{enumerate}

With respect to Conjecture~(\ref{conj1}), we know nowadays that the genus
can be arbitrarily larger than the rank (for non-hyperbolic manifolds see
\cite{chultens}; for
hyperbolic manifolds see \cite{toreador}).

With respect to Conjecture~(\ref{conj2}), one needs to establish for a
given manifold,
first, if there are indeed covering spaces that lower the genus, and,
then, to determine how large is the lowering. For hyperbolic
manifolds very few examples of genus-lowering covering spaces are known
(see \cite{chalen}, Section~4.5). For non-hyperbolic manifolds we can
distinguish 
cases:

\begin{enumerate}
\item If $\pi_1(M)$ is finite and non-trivial, then the universal
  cover of $M$ lowers the genus.

\item If $M$ is a torus bundle over $S^1$, we show in this work that only in
  special cases $M$ admits genus-lowering covering spaces.

\item For Seifert manifolds with orbit surface of genus $g$,
  
\begin{enumerate}
\item If $g\neq0$, there are only few examples of genus-lowering
  covering spaces (see Section~\ref{seifert} and \cite{jair}).
\item If $g=0$, there are examples, but it is an open problem to
  determine all possible genus-lowering covering spaces.
\end{enumerate}

\item If $M$ is a graph-manifold, it is an open problem to
  determine if there are genus-lowering covering spaces of $M$.

\end{enumerate}

The paper is organized as follows. In Section~\ref{sec1}, after some algebraic
remarks, we determine the structure of the covering spaces of surface
bundles, namely, we show that any covering space of surface bundles is
a product of a covering of fibers followed by a power covering
(Corollary~\ref{coro23}). This reduces the problem of finding
coverings that lower the genus to finding either power coverings or
fiber coverings that lower the genus. We also enlist results on torus bundles
from~\cite{saku1} that 
are used throughout the paper. In Section~\ref{sec3} we show that power
coverings of torus bundles do not lower the genus. In
Section~\ref{sec4} we determine the structure of the coverings of
fibers of torus bundles (Theorem~\ref{lemma42}) and characterize the  examples of coverings
of fibers of torus bundles that lower the genus
(Theorem~\ref{thm46}). Of interest is Corollary~\ref{coro48} where we explicitly
describe
the subgroups of the fundamental group of the torus which correspond to finite covering
spaces. We include in
Section~\ref{seifert} the 
examples of covering spaces of Seifert manifolds that lower the genus
mentioned above. These coverings of Seifert
manifolds are cyclic, and give examples for the remark in
Section~4.6 of~\cite{chalen}.

\section{Preliminaries.}
\label{sec1}

If $X$ is a set, we write $S(X)$ for the symmetric group on the
symbols in X. If $\#(X)=n$, we write $S_n=S(X)$. If $G\leq S(X)$, and
$i\in X$, we write $St_G(i)=\{\sigma\in G:\sigma(i)=i\}$; also write $St(i)=St_{S(X)}(i)$.

\begin{lemma}\label{normal}
	Assume that $G\leq S_n$ is a transitive group and $K\vartriangleleft
        G.$ Let $A_1, ..., A_m$ be the orbits of~$K$. Then for each
        $i\in\{1,\dots,m\}$ and each
        $\sigma\in G$, $\sigma\cdot A_i=A_s$ for some $s$. In
        particular $\#A_1=\cdots = \#A_m.$
\end{lemma}
\begin{proof}
For $\sigma \in G$ and $i\in\{1,\dots,m\}$,

$$
\begin{array}{ccll}
K\cdot\left(\sigma\cdot A_i\right) &=& (K\cdot\sigma)\cdot A_i\\
&=&\left(\sigma\cdot K\right)\cdot A_i, &\textrm{ for, $K$ is normal,}\\
&=&\sigma\cdot \left(K\cdot A_i\right)\\ 
&=&\sigma\cdot A_i, &\textrm{ for, $A_i$ is an orbit of $K.$}\\
\end{array}
$$

Thus $\sigma\cdot A_i$ is a union of orbits of $K$.
Since $A_i$ is an orbit, $A_i\neq\emptyset$. Pick some $a\in A_i$; then
$\sigma(a)\in A_s$ for some $s$, and $A_s\subset\sigma\cdot A_i$.

Assume that there is a $t$ such that $A_t\subset \sigma\cdot A_i$. Choose $b\in A_i$ such that $\sigma(b)\in A_t$.
Since there is a $\tau_1\in K$ such that $\tau_1(a)=b$, then
$\sigma(b)=\sigma(\tau_1(a))=\tau_2(\sigma(a))$ for some other
$\tau_2\in K$ for, $K$ is normal. Since $A_s$ is an orbit of $K$, then
$\tau_2(\sigma(a))\in A_s$. Then $A_s\cap A_t\neq\emptyset$, and
therefore $A_s=A_t$. We conclude that $\sigma\cdot A_i=A_s$.

Since $G$ is transitive, for each $s$ there is a $\sigma\in G$ such
that $\sigma\cdot A_1=A_s$. It follows that~$\#A_1=\cdots =\#A_m.$
\end{proof}

\begin{lemma} 
  \label{lemma12}
  Assume that $G\leq S_n$ is a transitive group and
  $K\vartriangleleft G.$ 
%Let $A_1, ..., A_m$ be the orbits of~$K$.
  Then there exist homomorphisms $q:G\to S_m$ and
  $\gamma:q^{-1}(St(1))\to S_{{n}/{m}}$ such that $St_G(1)\subset
  q^{-1}(St(1))$, and $q(K)=1$, and $\gamma|K$ is transitive.
\end{lemma}
\begin{proof}
  By Lemma~\ref{normal}, $G$ is imprimitive with the orbits of~$K$, $A_1,
  ..., A_m$, a set of  imprimitivity blocks. We
  assume that $1\in A_1$. Then we 
  have these 
  homomorphisms: $q:G\rightarrow S(\{A_1,\dots,A_m\})=S_m$ which is induced by
  the quotient $p:A_1\cup\cdots\cup A_m\rightarrow \{A_1,\dots,A_m\}$
    such that $p(a)=A_i\Leftrightarrow a\in A_i$, and
    $\gamma:q^{-1}(St(A_1))\to S(A_1)=S_{{n}/{m}}$ which is given by
    restriction $\gamma(\sigma)=\sigma|A_1$.
\end{proof}

\subsection{Coverings of surface bundles.}
\begin{corollary}
\label{coro23}
  Let  $F\hookrightarrow M\rightarrow S^1$ be a surface bundle over $S^1$, and
  let $\varphi:\widetilde{M}\to M$ be an~$n$-fold covering space.
  Then there
  is a commutative diagram of covering spaces
  of surface bundles over $S^1$
%QQQQQQQQQQQQQQQQQQQQQQQQQQQQQQQQQQQQQQQQQQQQQQQQQQQQQQQQQQQQQQQQQ
%QQQQQQQQQQQQQQQQQQQQQQQQQQQQQQQQQQQQQQQQQQQQQQQQQQQQQQQQQQQQQQQQQ
%QQQQQQQQQQQQQQQQQQQQQQQQQQQQQQQQQQQQQQQQQQQQQQQQQQQQQQQQQQQQQQQQQ
$$
\begin{tikzcd}
\widetilde M  \arrow[swap]{dd}{\varphi} \arrow{rd}{\varphi_\gamma} \\% & B \\
  & N \arrow{ld}{\varphi_q} \\
M
\end{tikzcd}
$$
  such that $\varphi_q$ and $\varphi_\gamma$ are $m$-fold and
  $n/m$-fold covering spaces, respectively, and
  $\varphi_q^{-1}(F)=F_1\sqcup\cdots\sqcup F_m$ with
  $\varphi_q|:F_i\rightarrow F$ a homeomorphism for $i=1,\dots,m$,  and
  $\varphi_\gamma^{-1}(\tilde F)$ is connected for $\tilde F$ any
  fiber of $N$.
  
\end{corollary}
\begin{proof}
  Recall that, if $F\hookrightarrow M\rightarrow S^1$ is a surface bundle, then
  $\pi_1(M)$ is isomorphic to a semi-direct product
  $\pi_1(F)\rtimes\mathbb{Z}$. In particular 
  $\pi_1(F)\lhd\pi_1(M)$. Then $\omega(\pi_1(F))\lhd Image(\omega)$
  where $\omega:\pi_1(M)\rightarrow S_n$ is the representation
  associated to $\varphi$. Lemma~\ref{lemma12} applies.

\end{proof}

\begin{remark}
Note that the coverings $\varphi_\gamma$ or $\varphi_q$ in
Corollary~\ref{coro23} might be homeomorphisms.
\end{remark}

 Let $F\hookrightarrow  M\rightarrow S^1$ be a surface
bundle. Then there is a homeomorphism $h:F\rightarrow F$, the
\emph{monodromy} of $M$, such that 
$$M=\frac{F\times I}{(x,0)\sim(h(x),1)}.$$
The \emph{infinite cyclic covering} of $M$, $u:F\times
\mathbb{R}\rightarrow M$, is the covering corresponding to the
subgroup~$\pi_1(F)\leq\pi_1(M)$, and has covering translations generated 
by~$t:F\times\mathbb{R}\rightarrow F\times\mathbb{R}$, such that 
$t(x,\lambda)=(h(x),\lambda+1)$.

\begin{remark} %\textsl{Remark.}
Let $F\hookrightarrow  M\rightarrow S^1$ be the surface bundle of
Corollary~\ref{coro23}. 
\begin{enumerate}
\item The covering $\varphi_q:N\rightarrow M$ of the corollary is constructed by
  taking $m$ copies  $F\times I\times\{1\},\dots,F\times
  I\times\{m\}$, and then glueing $((x,0),i)\sim((h(x),1),i+1)$ where the
  indices are taken mod~$m$; that is,
  $$N=
  \frac{F\times I\times\{1\}\sqcup\cdots\sqcup F\times
    I\times\{m\}}{((x,0),i)\sim((h(x),1),i+1)}.$$
  Note that $N$ is an $F$-bundle with monodromy $h^m$.

\item \label{psi} The covering $\varphi_\gamma:\widetilde M\rightarrow
  N$ of the corollary is
  constructed by 
  taking the covering space $\psi:\widetilde F\rightarrow F$ associated
  to $\omega|\pi_1(F):\pi_1(F)\rightarrow S_{{n}/{m}}$, and some
  monodromy $\widetilde g:\widetilde F\rightarrow \widetilde F$ such
  that the diagram
%QQQQQQQQQQQQQQQQQQQQQQQQQQQQQQQQQQQQQQQQQQQQQQQQQQQQQQQQQQQQQQQQQQQQ
%QQQQQQQQQQQQQQQQQQQQQQQQQQQQQQQQQQQQQQQQQQQQQQQQQQQQQQQQQQQQQQQQQQQQ
%QQQQQQQQQQQQQQQQQQQQQQQQQQQQQQQQQQQQQQQQQQQQQQQQQQQQQQQQQQQQQQQQQQQQ
$$
\begin{tikzpicture}[baseline= (a).base]
  \node[scale=1.15] (a) at (0,0){
    \begin{tikzcd}
      \tilde F \arrow{r}{\tilde g} \arrow{d}[swap]{\psi} & \tilde F \arrow{d}{\psi}\\
      F\arrow{r}{g} & F\\
    \end{tikzcd}
  };
\end{tikzpicture}
$$
 commutes, where $g$ is the monodromy of $N$. And
  $$\widetilde M=\frac{\widetilde F\times I}{(x,0)\sim(\widetilde
    g(x),1)}.$$
  The covering projection is obtained from $\psi\times 1$.

\item The covering $\varphi_q:N\rightarrow M$ in the corollary is
  characterized by a commutative diagram of covering spaces
%QQQQQQQQQQQQQQQQQQQQQQQQQQQQQQQQQQQQQQQQQQQQQQQQQQQQQQQQQQQQQQQQQQQQ
%QQQQQQQQQQQQQQQQQQQQQQQQQQQQQQQQQQQQQQQQQQQQQQQQQQQQQQQQQQQQQQQQQQQQ
%QQQQQQQQQQQQQQQQQQQQQQQQQQQQQQQQQQQQQQQQQQQQQQQQQQQQQQQQQQQQQQQQQQQQ
$$
\begin{tikzcd}
F\times\mathbb{R}  \arrow{dd}[swap]{u} \arrow{rd}{v} \\
  & N \arrow{ld}{\varphi_q} \\
M
\end{tikzcd}
$$
  where  $u$ and $v$ are infinite cyclic
  coverings. We call this type of covering space a \emph{power covering}.

\item The covering $\varphi_\gamma:\tilde M\rightarrow N$ in the corollary is
  characterized by a commutative diagram of covering spaces
%QQQQQQQQQQQQQQQQQQQQQQQQQQQQQQQQQQQQQQQQQQQQQQQQQQQQQQQQQQQQQQQQQQQQ
%QQQQQQQQQQQQQQQQQQQQQQQQQQQQQQQQQQQQQQQQQQQQQQQQQQQQQQQQQQQQQQQQQQQQ
%QQQQQQQQQQQQQQQQQQQQQQQQQQQQQQQQQQQQQQQQQQQQQQQQQQQQQQQQQQQQQQQQQQQQ  
$$
\begin{tikzcd}
\tilde F \times\mathbb{R}\arrow{r}{\psi\times1} \arrow{d}[swap]{\tilde
u} & F\times\mathbb{R} \arrow{d}{u}\\
\tilde M\arrow{r}{\varphi_\gamma} & M\\
\end{tikzcd}
$$
    where $\psi$ is as in (\ref{psi}), and $u$ and $\tilde u$ are infinite
    cyclic coverings. We call this type of covering space
    a \emph{covering of fibers}. 

Notice that, since $\psi\times1$ sends each fundamental region of
$\tilde u$ onto a fundamental region of $u$, the covering
$\psi\times1$ commutes with the covering translations of $\tilde u$ and $u$.
  \end{enumerate}
\end{remark}

\subsection{Torus bundles}
\label{sec22}
Let $M$ be a torus bundle over $S^1$. Then $$M\cong \frac{T^2\times
  I}{(x,0)\sim(A(x),1)},$$ 
for some homeomorphism $A:T^2\to T^2$. We write
$M=M_A$ for the torus bundle with monodromy $A$.  
Throughout this paper we fix a basis~$\pi_1(T^2)\cong\langle
x,y:[x,y]\rangle$. With respect to this basis, the homeomorphism $A$ can be
identified with an integral invertible matrix, namely, its induced isomorphism
$A_\#:\pi_1(T)\rightarrow\pi_1(T)$. 
We
consider here only orientable torus bundles, that is, $A\in SL(2,\mathbb{Z})$.

It is known that the Heegaard genus of $M_A$
is~two or~three. Also the fundamental group of $M_A$ is a semi-direct product
$$\pi_1(M_A)\cong\pi_1(T^2)\rtimes\mathbb{Z}\cong\langle
x,y,t:x^t=x^\alpha y^\gamma,y^t=x^\beta y^\delta,[x,y]=1\rangle,$$
where 
$A=\left(
  \begin{smallmatrix}
    \alpha & \beta \\
    \gamma & \delta
  \end{smallmatrix}
\right)$.
The first homology group of $M_A$ is of the form
$$H_1(M_A)\cong\mathbb{Z}\oplus Coker(A-I)\cong\mathbb{Z}\oplus \mathbb{Z}_{n_1}\oplus
\mathbb{Z}_{n_2}$$
where $n_1|n_2$.

By \cite{saku1}, the following conditions are equivalent:
\begin{itemize}
\item $M_A$ is a double branched covering of $S^3$.
\item $n_1= 1,  2$. 
\item $A$ is conjugate to
  $\left(
    \begin{smallmatrix}
      -1 & -a \\
      b & ab-1
    \end{smallmatrix}
  \right)$
in $GL(2,\mathbb{Z})$ for some integers $a$ and~$b$.
\end{itemize}

Notice that 
$\left(
	\begin{smallmatrix}
	-1 & -a \\
	b & ab-1
	\end{smallmatrix}
	\right)=
\left(
	\begin{smallmatrix}
	0 & 1 \\
	1 & 0
	\end{smallmatrix}
	\right)
\left(
	\begin{smallmatrix}
	-1 & -b \\
	a & ab-1
	\end{smallmatrix}
	\right)^{-1}
\left(
	\begin{smallmatrix}
	0 & 1 \\
	1 & 0
	\end{smallmatrix}
	\right)
$; that is, the integers~$a$ and~$b$ are interchangeable.

If the Heegaard genus of $M_A$ is two, then $M_A$ is a double
branched 
covering of the 3-sphere. Also, $g(M_A)=2$ if and only if $A$ is
conjugate to 
$\left(
	\begin{smallmatrix}
	-1 & -1 \\
	b & b-1
	\end{smallmatrix}
	\right)
$ in $GL(2,\mathbb{Z})$ (see~\cite{saku1}).

If 
$A=\left(
	\begin{smallmatrix}
	-1 & -a \\
	b & ab-1
	\end{smallmatrix}
	\right)$,
write $M_{a,b}=M_A$. Then the numbers $a$ and $b$ are complete invariants of
$M_A$. That is, $M_{a_1,b_1}\cong M_{a_2,b_2}$ if and only if the sets
$\{a_1,b_1\}=\{a_2,b_2\}$, except for $M_{1,6}\cong M_{2,3}$ (see \cite{saku1},
Theorem~4).

%%%%%%%%%%%%%%%%%%%%%%%%%%%%%%%%%%%%%%%%%%%%%%%%%%%%%%%%%%%%%%%%%%%%%%%%%%%%
%%%%%%%%%%%%%%%%%%%%%%%%%%%%%%%%%%%%%%%%%%%%%%%%%%%%%%%%%%%%%%%%%%%%%%%%%%%%
%%%%%%%%%%%%%%%%%%%%%%%%%%%%%%%%%%%%%%%%%%%%%%%%%%%%%%%%%%%%%%%%%%%%%%%%%%%%
%%%%%%%%%%%%%%%%%%%%%%%%%%%%%%%%%%%%%%%%%%%%%%%%%%%%%%%%%%%%%%%%%%%%%%%%%%%%
%%%%%%%%%%%%%%%%%%%%%%%%%%%%%%%%%%%%%%%%%%%%%%%%%%%%%%%%%%%%%%%%%%%%%%%%%%%%
%%%%%%%%%%%%%%%%%%%%%%%%%%%%%%%%%%%%%%%%%%%%%%%%%%%%%%%%%%%%%%%%%%%%%%%%%%%%

\section{Power coverings of torus bundles.} 
\label{sec3}
In this section we prove
\begin{theorem} 
\label{thm31}
Let $\eta:\widetilde M\rightarrow M$ be a finite
  power covering space of torus bundles. If 
  the Heegaard genus of $M$ is three, then the Heegaard genus of
  $\widetilde M$ is also three.
\end{theorem}
\begin{proof}
Let us assume that $M=M_A$ for some $A\in SL(2,\mathbb{Z})$, and fix
$\eta:\widetilde M\rightarrow M_A$ an $n$-fold  power covering. Then
$\widetilde M\cong M_{A^n}$. In the following lemmata we check that, if the Heegaard genus of
$M_A$ is three and~$n\geq2$, then the genus of~$M_{A^n}$ is at least three, by
computing the rank of~$\pi_1(M_{A^n})$ or of~$H_1(M_{A^n})$, which are lower
bounds for the genus. 
\end{proof}

Recall that the first homology group of
$M_{A^n}$ is
$$H_1(M_{A^n})\cong\mathbb{Z}\oplus Coker(A^n-I).$$
We repeatedly use the fact that
$$A^n-I=(A-I)(A^{n-1}+\cdots+A+I)$$

For reference purposes, we upgrade the following easy remark to a lemma.

\begin{lemma}
\label{lemma20}
Let $C$ be an $m\times m$ integral matrix such that there is an
integer $s\neq\pm1$ with $s|C_{i,j}$ for each $i,j$. Then the rank of
$Coker(C)$ is~$m$.
\end{lemma}
\begin{proof}
The Smith normal form of $C$ is
$\left(\begin{smallmatrix}
t_1\\
 &\ddots\\
& & t_m\\
\end{smallmatrix}\right)$
with $t_i|t_{i+1}$. Since~$s|C_{i,j}$ for each $i,j$, and
$t_1$ is the greatest common divisor of the $C_{i,j}$, then~$s|t_1$, 
and the lemma follows.

\end{proof}

\begin{lemma}
\label{lemma21}
If $M_A$ is not a double branched cover of the 3-sphere, and $n\geq2$, then the
rank of $H_1(M_{A^n})$ is three. 
\end{lemma}
\begin{proof}
By \cite{saku1}, the Smith normal form of $A-I$ is 
$\left(
	\begin{smallmatrix}
	n_1 & 0 \\
	0 & n_2
	\end{smallmatrix}
	\right)
$ with $n_1|n_2$, and $n_1\neq1,2$. 

Then $H_1(M_{A^n})\cong\mathbb{Z}\oplus 
 Coker\left((A-I)(A^{n-1}+\cdots+A+I)\right)\cong\mathbb{Z}\oplus  Coker(
 \left(
   \begin{smallmatrix}
     n_1 & 0 \\
     0 & n_2
   \end{smallmatrix}
 \right) B)
$ for some matrix~$B$. Now $n_1$ divides all entries of the matrix~$\left(
  \begin{smallmatrix}
    n_1 & 0 \\
    0 & n_2
  \end{smallmatrix}
\right) B$, and Lemma~\ref{lemma20} applies.
\end{proof}

We assume now that $M_A$ is a double branched covering of $S^3$. Then
we may assume 
that~$A=
\left(
  \begin{smallmatrix}
    -1 & -a \\
    b & ab-1
  \end{smallmatrix}
\right)$. 
If the Heegaard genus of $M_A$ is three, then $|a|,|b|\neq1$.

\begin{lemma}
\label{zero}
If one of $a$ or $b$ is zero, and $n\geq2$, then the Heegaard genus of
$M_{A^n}$ is three.
\end{lemma}
\begin{proof}
We may assume that $b=0$. Notice that
$$A^n=\left\{
\begin{array}{cl}
  \left(\begin{array}{cc}
      -1&-na\\
      0&-1\\
    \end{array}\right),
  &\textrm{ if } n \textrm{ is odd}\\ 
  \left(\begin{array}{cc}
      1&na\\
      0&1\\
    \end{array}\right),
  &\textrm{ if } n \textrm{ is even}\\ 
\end{array}\right.
$$

If also $a=0$, then 
$A^n-I=\left(
  \begin{smallmatrix}
    0 & 0 \\
    0 & 0
  \end{smallmatrix}
\right)
$
or
$\left(
  \begin{smallmatrix}
    -2 & 0 \\
    0 & -2
  \end{smallmatrix}
\right)
$.
Then Lemma~\ref{lemma20} applies,
and the rank of $H_1(M_{A^n})$ is three. Assume then that $a\neq0$.

\begin{sublemma}
\label{lemma23}
If 
$B=\left(\begin{smallmatrix}
    1&\alpha\\
    0&1\\
  \end{smallmatrix}\right)
$, and $|\alpha|\geq2$, then the rank of $H_1(M_{B})$ is three.
\end{sublemma}
\begin{proof}[Proof of Claim~\ref{lemma23}]
The Smith normal form of $B-I$ is 
$\left(\begin{smallmatrix}
    \alpha & 0\\
    0 & 0\\
  \end{smallmatrix}\right)
$.
Then Lemma~\ref{lemma20} applies to $H_1(M_B)\cong\mathbb{Z}\oplus Coker(B-I)$.
\end{proof}

\begin{sublemma}
\label{lemma24}
If
$B=\left(\begin{smallmatrix}
    -1 & -\alpha\\
    0 & -1\\
  \end{smallmatrix}\right)
$, and $|\alpha|\geq2$, then 
the rank of $\pi_1(M_B)$ is three.
\end{sublemma}
\begin{proof}[Proof of Claim~\ref{lemma24}]
The fundamental group of $M_B$ is
$$\pi_1(M_B)=\langle
x,y,t:txt^{-1}x,tyt^{-1}yx^\alpha,xyx^{-1}y^{-1}\rangle.$$
Write $r_1=txt^{-1}x$, $r_2=tyt^{-1}yx^\alpha$, and
$r_3=xyx^{-1}y^{-1}$. 

Also
write $d_1=\partial/\partial x$, $d_2=\partial/\partial y$, and
$d_3=\partial/\partial t$, the Fox derivatives. Then
$$d_1(r_2)=tyt^{-1}yP,\quad d_2(r_2)=tyt^{-1}+t,\quad 
d_3(r_2)=-tyt^{-1}+1$$ 
where $P=1+x+\cdots+x^{\alpha-1}$ if $\alpha>0$, and
$P=-(x^{-1}+x^{-2}+\cdots+x^\alpha)$ if 
$\alpha<0$. 
Also
$$d_1(r_1)=txt^{-1}+t,\quad d_2(r_1)=0,\quad d_3(r_1)=-txt^{-1}+1$$
and
$$d_1(r_3)=-xyx^{-1}+1,\quad
d_2(r_3)=-xyx^{-1}y^{-1}+x,\quad
d_3(r_3)=0$$

We define a ring homomorphism $\rho:\mathbb{Z}\pi_1(M_B)\rightarrow\mathbb{Z}_\alpha$, from the
group ring of $\pi_1(M_B)$ into the ring of integers
$\pmod\alpha$. First define
$$\rho(x)=1,\quad \rho(y)=1,\quad \rho(t)=-1.$$

Extending $\rho:\pi_1(M_B)\rightarrow\mathbb{Z}_\alpha$
multiplicatively, we see that $\rho(r_i)=1$ for $i=1,2,3$. 

Then we can extend $\rho$ to the ring group,
$\rho:\mathbb{Z}\pi_1(M_B)\rightarrow\mathbb{Z}_\alpha$, 
as a ring homomorphism sending 1 to 1. Also~$\rho$ sends all Fox derivatives
of the relators $r_1$, $r_2$ and $r_3$ into zero.

By \cite{Lu}, we conclude that rank$(\pi_1(M_B))=3$.
\end{proof}

To finish the proof of Lemma~\ref{zero}, we note that
claims~\ref{lemma23} and \ref{lemma24} give the number 3 as a lower bound for the genus
of $M_{A^n}$ when $n$ is even or odd, respectively. It follows that the genus
of $M_{A^n}$ is three.

\end{proof}

The remaining case is $|a|,|b|\geq2$ for
$A=
\left(
  \begin{smallmatrix}
    -1 & -a \\
    b & ab-1
  \end{smallmatrix}
\right)$. 

\begin{lemma}
\label{lemma25}
If $|a|,|b|\geq2$, and $n\geq2$, then the rank of
$H_1(M_{A^n})$ is three.
\end{lemma}
\begin{proof}
If $a$ and $b$ are both even integers, then the Smith normal form of
$A-I$ is
$\left(\begin{smallmatrix}
    -2 & 0\\
    0 & ab/2-2\\
  \end{smallmatrix}\right)
$, and, as in the proof of Lemma~\ref{lemma21}, Lemma~\ref{lemma20} applies to $Coker(A^n-I)$.

Then assume that $a$ and $b$ are not both even. Note that, then, $|ab|\geq6$.

We define $f(0)=f(1)=1$, and

$$f(n)=\left\lbrace
\begin{array}{cc}
f(n-1)-f(n-2), & \textrm{ if $n$ is odd}\\
abf(n-1)-f(n-2), & \textrm{ if $n$ is even}
\end{array}
\right.
$$
For convenience we define $f(-1)=0$.
\begin{sublemma}
\label{lemma26}
If $n\geq 1,$ then $A^n= \left(\begin{array}{cc}
	-f(2n-2)  & -af(2n-1) \\
	bf(2n-1) & f(2n)
	\end{array}
\right)$.
\end{sublemma}
\begin{proof}[Proof of Claim~\ref{lemma26}]
This follows by an easy induction on $n$.
\end{proof}

\begin{sublemma}
\label{lemma27}
If $n\geq 2,$ then $A^n+A^{n-1}+\cdots+A+I=$
$$ \left\lbrace
\begin{array}{cc}
f(n)\left(\begin{array}{cc}
	-abf(n-2)  & -af(n-1) \\
	bf(n-1) & abf(n)
	\end{array}
\right), & \textrm{ if $n$ is odd}\\
& \\
f(n)\left(\begin{array}{cc}
	-f(n-2)  & -af(n-1) \\
	bf(n-1) & f(n)
	\end{array}
\right), & \textrm{ if $n$ is even.}
\end{array}
\right.
$$
\end{sublemma}
\begin{proof}[Proof of Claim~\ref{lemma27}]
One can check directly that the lemma holds for $n=1,2$, and~$3$.

Assume that $n\geq3$.

\textsl{First Case:} ``$n$ is even''. Say $n=2k$.

Then, by Lemma~\ref{lemma26},

\noindent
\resizebox{\linewidth}{!}{%
$
\begin{array}{rl}
  A^n=& A^kA^k\\
  =&
  \left(\begin{array}{cc}
      -f(n-2)  & -af(n-1) \\
      bf(n-1) & f(n)
    \end{array}
  \right)^2\\
  =&
  \left(\begin{array}{cc}
      f(n-2)^2-abf(n-1)^2  & af(n-2)f(n-1)-af(n-1)f(n) \\
      -bf(n-2)f(n-1)+bf(n-1)f(n)) & -abf(n-1)^2+f(n)^2\\
    \end{array}
  \right).\\
\end{array}
$
}

Write $B=A^n+A^{n-1}+\cdots+A+I$. Then, by induction on $n$,
$$
B= A^n+
\begin{array}{cc}
  f(n-1)\left(\begin{array}{cc}
      -abf(n-3)  & -af(n-2) \\
      bf(n-2) & abf(n-1)
    \end{array}
  \right)
\end{array}
$$
for, $n-1$ is odd.
Equating elements

$$\begin{array}{rl}
B_{1,1}=&f(n-2)^2-abf(n-1)^2-abf(n-1)f(n-3)\\
=&f(n-2)^2-abf(n-1)(f(n-1)+f(n-3))\\
=&f(n-2)^2-abf(n-1)f(n-2), \quad\textrm{ for, $n-1$ is odd}\\
=&-f(n-2)(abf(n-1)-f(n-2))\\
=&-f(n-2)f(n), \quad\textrm{ for, $n$ is even}\\
\\
B_{1,2}=&af(n-2)f(n-1)-af(n-1)f(n)-af(n-1)f(n-2)\\
=&-af(n-1)f(n)\\
\\
B_{2,1}=&-bf(n-2)f(n-1)+bf(n-1)f(n)+bf(n-2)f(n-1)\\
=&bf(n-1)f(n)\\
\\
B_{2,2}=&-abf(n-1)^2+f(n)^2+abf(n-1)^2\\
=&f(n)^2,\\
\end{array}
$$
That is
$$A^n+A^{n-1}+\cdots+A+I=
\begin{array}{cc}
  f(n)\left(\begin{array}{cc}
      -abf(n-2)  & -af(n-1) \\
      bf(n-1) & abf(n)
    \end{array}
  \right).
\end{array}
$$

\textsl{Second Case:} ``$n$ is odd''. Say $n=2k+1$.

Now, by Lemma~\ref{lemma26},

\resizebox{\linewidth}{!}{%
$
\arraycolsep=1.8pt\def\arraystretch{1.4}
\begin{array}{rl}
  A^n=& A^{k+1}A^k\\
  =&
  \left(\begin{array}{cc}
      -f(n-1)  & -af(n) \\
      bf(n) & f(n+1)
    \end{array}
  \right)
  \left(\begin{array}{cc}
      -f(n-3)  & -af(n-2) \\
      bf(n-2) & f(n-1)
    \end{array}
  \right)\\

=&
\left(\begin{array}{cc}
    %\textstyle
    f(n-3)f(n-1)-abf(n-2)f(n)  &
    %\textstyle 
    af(n-1)f(n-2)-af(n-1)f(n) \\
    %\scriptstyle
    -bf(n-3)f(n)+bf(n-2)f(n+1) & 
    %\scriptstyle
    -abf(n-2)f(n)+f(n-1)f(n+1)\\
  \end{array}
\right).\\
\end{array}
$
}

\noindent
Write $B=A^n+A^{n-1}+\cdots+A+I$. Then, by induction on $n$,
$$
B= A^n+
\begin{array}{cc}
  f(n-1)\left(\begin{array}{cc}
      -f(n-3)  & -af(n-2) \\
      bf(n-2) & f(n-1)
    \end{array}
  \right)
\end{array}
$$
for, $n-1$ is even.
Equating
$$\begin{array}{rl}
B_{1,1}=&f(n-3)f(n-1)-abf(n-2)f(n)-f(n-3)f(n-1)\\
=&-abf(n-2)f(n)\\
\\
B_{1,2}=&af(n-2)f(n-1)-af(n-1)f(n)-af(n-2)f(n-1)\\
=&-af(n-1)f(n)\\
\\
B_{2,1}=&-bf(n-3)f(n)+bf(n-2)f(n+1)+bf(n-2)f(n-1)\\
=&-bf(n-3)f(n)+bf(n-2)(f(n+1)+f(n-1))\\
=&-bf(n-3)f(n)+bf(n-2)abf(n) \quad\textrm{ for, $n+1$ is even}\\
=&bf(n)(abf(n-2)-f(n-3))\\
=& bf(n)f(n-1) \quad\textrm{for, $n-1$ is even}
\\
B_{2,2}=&-abf(n-2)f(n)+f(n-1)f(n+1)+f(n-1)^2\\
=&-abf(n-2)f(n)+f(n-1)(f(n+1)+f(n-1))\\
=&-abf(n-2)f(n)+f(n-1)abf(n) \quad\textrm{for, $n+1$ is even}\\
=&abf(n)(f(n-1)-f(n-2))\\
=&abf(n)^2  \quad\textrm{for, $n$ is odd}
\end{array}
$$
That is
$$A^n+A^{n-1}+\cdots+A+I=
\begin{array}{cc}
  f(n)\left(\begin{array}{cc}
      -abf(n-2)  & -af(n-1) \\
      bf(n-1) & abf(n)
    \end{array}
  \right).
\end{array}
$$

\end{proof}

\begin{sublemma} 
\label{lemma28}
If $|ab|\geq6$, and $n\geq2$, then $|f(n)|>1$.
\end{sublemma}
\begin{proof}[Proof of Claim~\ref{lemma28}]

Consider the `generating function' $G(z)=\sum_{i=0}^\infty
f(i)z^i$. The even summands of $G(z)$ are given by
$\frac12(G(z)+G(-z))$, and the odd summands by $\frac12(G(z)-G(-z))$.
Then, by the definition of the sequence~$f$,
$$\frac12(G(z)+G(-z))-\frac12abz(G(z)-G(-z))+\frac12z^2(G(z)+G(-z))=1$$
and
$$\frac12(G(z)-G(-z))-\frac12z(G(z)+G(-z))+\frac12z^2(G(z)-G(-z))=0,$$
that is,
$$(1-abz+z^2)G(z)+(1+abz+z^2)G(-z)=2$$
and
$$(1-z+z^2)G(z)+(-1-z-z^2)G(-z)=0.$$
Thus
$$G(-z)=\frac{1-z+z^2}{1+z+z^2}G(z)$$
and, substituting, % in the first equation,
$$G(z)=\frac{1+z+z^2}{1+(2-ab)z^2+z^4}.$$

Write
$$\displaystyle\varphi=\frac{-(2-ab)+\sqrt{(2-ab)^2-4}}{2},\quad
\displaystyle\hat\varphi=\frac{-(2-ab)-\sqrt{(2-ab)^2-4}}{2},$$
which are the solutions of $1+(2-ab)t+t^2=0$. Notice that
$\varphi\hat\varphi=1$, 
and $\varphi+\hat\varphi=-(2-ab)$. Then
$$G(z)=\frac{1+z+z^2}{(1-\varphi z^2)(1-\hat\varphi z^2)}.$$
Since $|ab|\geq6$,  we have $\varphi\neq\hat\varphi$.
We can write 
$$\alpha=-\frac{1+\varphi+\varphi z}{\hat\varphi-\varphi},\quad 
\beta=\frac{1+\hat\varphi+\hat\varphi z}{\hat\varphi-\varphi},$$
then
$$G(z)=\frac{\alpha}{1-\varphi z^2}+\frac{\beta}{1-\hat\varphi z^2}.$$
Now
$$\frac{1}{1-\varphi z^2}=\sum_{i=0}^\infty \varphi^i z^{2i}, \quad
\frac{1}{1-\hat\varphi z^2}=\sum_{i=0}^\infty \hat\varphi^i z^{2i},$$
then
$$G(z)=\sum_{i=0}^\infty\frac{\hat\varphi^i-\varphi^i+\hat\varphi^{i+1}-\varphi^{i+1}}{\hat\varphi-\varphi}z^{2i}+\sum_{i=0}^\infty\frac{\hat\varphi^i-\varphi^i}{\hat\varphi-\varphi}z^{2i+1},$$
that is,
$$f(n)=\left\{
  \begin{array}{ll}
    \displaystyle\frac{\hat\varphi^k-\varphi^k+\hat\varphi^{k+1}-\varphi^{k+1}}{\hat\varphi-\varphi}
    &\textrm{ if } n=2k, n\geq0\\
    \displaystyle\frac{\hat\varphi^k-\varphi^k}{\hat\varphi-\varphi} &\textrm{ if } n=2k+1,n\geq1 
  \end{array}\right.
$$
%%%%%%%%%%%%%%%%%%%%%%%%%%%%%%%%%%%%%%%%%%%%%%%%%%%%%%%%%%%%%%%%%%%%%%%%%%%%%%%%%%%%
%%%%%%%%%%%%%%%%%%%%%%%%%%%%%%%%%%%%%%%%%%%%%%%%%%%%%%%%%%%%%%%%%%%%%%%%%%%%%%%%%%%%
%%%%%%%%%%%%%%%%%%%%%%%%%%%%%%%%%%%%%%%%%%%%%%%%%%%%%%%%%%%%%%%%%%%%%%%%%%%%%%%%%%%%
%
Recall that
$(x^m-y^m)/(x-y)=\sum_{i=0}^{m-1}x^{m-1-i}y^i.$
%=\sum_{i=0}^{m-1}y^{m-1-i}x^i$. 
We have that, by definition,~$\varphi>\hat\varphi$.
Also, since $\varphi\hat\varphi=1$, either both $\varphi,\hat\varphi>0$ or both
$\varphi,\hat\varphi<0$. 

Assume that $n$ is odd, say, $n=2k+1$. Then if both $\varphi,
\hat\varphi>0$, it follows that $\varphi>1$, and
$$
f(n)
=\displaystyle\frac{\hat\varphi^k-\varphi^k}{\hat\varphi-\varphi} 
=\displaystyle\sum_{i=0}^{k-1}\varphi^{k-1-i}\hat\varphi^i>\varphi^{k-1}>1,
\textrm{ if $k>1$.}
$$
If both $\varphi, \hat\varphi<0$, then $\hat\varphi<-1$, and
$$
f(n)
=\displaystyle\frac{\hat\varphi^k-\varphi^k}{\hat\varphi-\varphi} \\
=\displaystyle\sum_{i=0}^{k-1}\hat\varphi^{k-1-i}\varphi^i=
\hat\varphi^{k-1}+\sum_{i=1}^{k-1}\hat\varphi^{k-1-i}\varphi^i.
$$

If $k$ is odd, then $f(n)>\hat\varphi^{k-1}>1$ when $k>1$. If $k$ is
even, then $f(n)<\hat\varphi^{k-1}<-1$ when $k>0$, and then $|f(n)|>1$.

In any case $|f(n)|>1$ for $n$ odd and $n\geq5$.
We check directly $|f(3)|=|2-ab|\geq4$ for, $|ab|\geq6$.

%%%%%%%%%%%%%%%%%%%%%%%%%%%%%%%%%%%%%%%%%%%%%%%%%%%%%%%%%%%%%%%%%%%%%%%%%%%%%%%%%%%%
%%%%%%%%%%%%%%%%%%%%%%%%%%%%%%%%%%%%%%%%%%%%%%%%%%%%%%%%%%%%%%%%%%%%%%%%%%%%%%%%%%%%
%%%%%%%%%%%%%%%%%%%%%%%%%%%%%%%%%%%%%%%%%%%%%%%%%%%%%%%%%%%%%%%%%%%%%%%%%%%%%%%%%%%%
%
Now if $n$ is even, say, $n=2k$, then
%and $\varphi$ and $\hat\varphi$ are both positive numbers,
$$
\arraycolsep=1.4pt\def\arraystretch{2.6}
\begin{array}{rl}
\displaystyle
f(n)=&\displaystyle\frac{\hat\varphi^k-\varphi^k}{\hat\varphi-\varphi}+\frac{\hat\varphi^{k+1}-\varphi^{k+1}}{\hat\varphi-\varphi}\\
=&\displaystyle\sum_{i=0}^{k-1}\hat\varphi^{k-1-i}\varphi^i+\sum_{j=0}^{k}\hat\varphi^{k-j}\varphi^j.
%>1,
%\textrm{ if $k>0$.}\\
\end{array}
$$

Thus, if $\varphi,\hat\varphi>0$, $f(n)>\varphi^{k}>1$ for $n$ even and $\geq2$. 

 If $\varphi$ and $\hat\varphi$  both are negative numbers, 
then
 $\hat\varphi<-1<\varphi<0$. Notice that,
 $\varphi<0$ implies $2-ab>0$. Since $|ab|\geq6$, we have $|2-ab|\geq4$.
Thus
$\sqrt{(2-ab)^2-4}\geq\sqrt{4^2-4}=\sqrt{12}$. 

Then
$$\hat\varphi=\frac{-(2-ab)-\sqrt{(2-ab)^2-4}}{2}\leq\frac{-4-\sqrt{12}}{2}=-2-\sqrt3<-3$$ 
and
since $1=\varphi\hat\varphi>\varphi(-3)$
we have that
$$0>\varphi>-1/3 \ \ \textrm{ and }\ \ \  \hat\varphi<-3.$$
We compute
%
%%%%%%%%%%%%%%%%%%%%%%%%%%%%%%%%%%%%%%%%%%%%%%%%%%%%%%%%%%%%%%%%%%
%%%%%%%%%%%%%%%%%%%%%%%%%%%%%%%%%%%%%%%%%%%%%%%%%%%%%%%%%%%%%%%%%%
%%%%%%%%%%%%%%%%%%%%%%%%%%%%%%%%%%%%%%%%%%%%%%%%%%%%%%%%%%%%%%%%%%

$$
\arraycolsep=1.4pt\def\arraystretch{2.6}
\begin{array}{rl}
\displaystyle
f(n)=&\displaystyle\sum_{i=0}^{k-1}\hat\varphi^{k-1-i}\varphi^i+\sum_{j=0}^{k}\hat\varphi^{k-j}\varphi^j\\
=&\displaystyle\varphi\hat\varphi\sum_{i=0}^{k-1}\hat\varphi^{k-1-i}\varphi^i+\sum_{j=0}^{k}\hat\varphi^{k-j}\varphi^j
\\
=&\displaystyle\varphi\hat\varphi\hat\varphi^{k-1}\sum_{i=0}^{k-1}\varphi^{2i}+\hat\varphi^{k}\sum_{j=0}^{k}\varphi^{2j} 
\\
=&\displaystyle\hat\varphi^{k}\sum_{i=0}^{k-1}\varphi^{2i+1}+\hat\varphi^{k}\sum_{j=0}^{k}\varphi^{2j} 
\\
=&\displaystyle\hat\varphi^{k}\sum_{i=0}^{2k}\varphi^{i}
\\
=&\displaystyle\hat\varphi^{k}\frac{\varphi^{2k+1}-1}{\varphi-1}.\\

\end{array}
$$
The distance from $\varphi$ to 1 is $|\varphi-1|<1+1/3=4/3$ for,
$0>\varphi>-1/3$; and the distance $|\varphi^{2k+1}-1|>1$ for,
$\varphi^{2k+1}<0$. Thus
$$|f(n)|=\displaystyle|\hat\varphi^{k}|\frac{|\varphi^{2k+1}-1|}{|\varphi-1|}>\frac34|\hat\varphi^{k}|>\frac34
3^k>1$$
for $k\geq1$.

Therefore, if $\varphi,\hat\varphi<0$, $|f(n)|>1$ for each $n$ even and
$n\geq2$. 

In any case, $|f(n)|>1$ for each $n\geq2$. 
\end{proof}

To finish the proof of Lemma~\ref{lemma25}, we have that, by
claims~\ref{lemma27} and~\ref{lemma28}, $A^n-I=f(n)B$ for some matrix $B$, and
$|f(n)|>1$ if~$n\geq2$. 
By Lemma~\ref{lemma20},
it follows that $H_1(M_{A^n})=\mathbb{Z}\oplus Coker(A^n-I)$ has rank
three for $n\geq2$.
\end{proof}

\begin{remark}
Notice that, following the proof of previous lemma, we see that for a torus
bundle $M\cong M_{1,b}$, the
power coverings of $M$ have genus three if $|b|\geq6$. %, or $b=0$.

\end{remark}

%%%%%%%%%%%%%%%%%%%%%%%%%%%%%%%%%%%%%%%%%%%%%%%%%%%%%%%%%%%%%%%%%%%%%%%%%%%%
%%%%%%%%%%%%%%%%%%%%%%%%%%%%%%%%%%%%%%%%%%%%%%%%%%%%%%%%%%%%%%%%%%%%%%%%%%%%
%%%%%%%%%%%%%%%%%%%%%%%%%%%%%%%%%%%%%%%%%%%%%%%%%%%%%%%%%%%%%%%%%%%%%%%%%%%%
%%%%%%%%%%%%%%%%%%%%%%%%%%%%%%%%%%%%%%%%%%%%%%%%%%%%%%%%%%%%%%%%%%%%%%%%%%%%
%%%%%%%%%%%%%%%%%%%%%%%%%%%%%%%%%%%%%%%%%%%%%%%%%%%%%%%%%%%%%%%%%%%%%%%%%%%%
%%%%%%%%%%%%%%%%%%%%%%%%%%%%%%%%%%%%%%%%%%%%%%%%%%%%%%%%%%%%%%%%%%%%%%%%%%%%

\section{Fiber coverings of torus bundles.}
\label{sec4}
Let $T_A\hookrightarrow M_A\rightarrow S^1$ be a torus bundle over $S^1$ with
monodromy~$A\in SL(2,\mathbb{Z})$. 
Write $\widetilde{M}_A=T_A\times\mathbb{R}$ for the
infinite cyclic covering $u:\widetilde{M}_A\rightarrow M_A$ corresponding to the
subgroup~$\pi_1(T_A)\leq\pi_1(M_A)$. 
We identify~$\pi_1(T_A)$ with $\pi_1(\widetilde{M}_A)$ through the
isomorphism induced by the inclusion $T_A\hookrightarrow \widetilde{M}_A$.
The group of covering transformations of $u$
is generated by $\varphi_A:\widetilde{M}_A\rightarrow\widetilde{M}_A$ given
by $\varphi_A(z,s)=(A(z),s+1)$. 
The induced homomorphism
$(\varphi_A)_\#:\pi_1(\widetilde{M}_A)\rightarrow\pi_1(\widetilde{M}_A)$
acts as the matrix~$A$, and we  abuse notation writing $\varphi_A=(\varphi_A)_\#$. Then we have an action of the ring group
$\mathbb{Z}\langle t\rangle$ on $\pi_1(\widetilde{M}_A)$ given by
$t\cdot c=\varphi_A(c)$ for
each~$c\in\pi_1(\widetilde{M}_A)$, where $\langle t\rangle$ is the infinite cyclic
group generated by~$t$. The structure of~$\mathbb{Z}\langle
t\rangle$-module on~$\pi_1(\widetilde{M}_A)$ obtained by
this action is
denoted by $H_A$. Notice that~$\pi_1(M_A)\cong H_A\rtimes_{\varphi_A}\langle
t\rangle$. 

For a covering of fibers of torus bundles, the notation $\eta:M_B\rightarrow
M_A$ is reserved, and implies the following statement: If $T_A$ is a fiber
of~$M_A$ and $T_B=\eta^{-1}(T_A)$, then the diagram
$$
  \begin{tikzcd}
    {T_B}\arrow{r}{B} \arrow[swap]{d}{\eta|} & {T_B} \arrow{d}{\eta|}\\
    T_A\arrow{r}{A} & T_A\\
  \end{tikzcd}
  $$  
commutes.

%QQQQQQQQQQQQQQQQQQQQQQQQQQQQQQQQQQQQQQQQQQQQQQQQQQQQQQQQQQQQQQQQQQQQQQQQQQ
%QQQQQQQQQQQQQQQQQQQQQQQQQQQQQQQQQQQQQQQQQQQQQQQQQQQQQQQQQQQQQQQQQQQQQQQQQQ
%QQQQQQQQQQQQQQQQQQQQQQQQQQQQQQQQQQQQQQQQQQQQQQQQQQQQQQQQQQQQQQQQQQQQQQQQQQ

\begin{theorem}
\label{lemma42}
	Let $L\leq H_A$ be an additive subgroup. Then the following are
        equivalent. 
\begin{enumerate}
\item $L$ is a $\mathbb{Z}\langle t\rangle$-submodule of $H_A$ of index $n$.
\item There are $B\in SL_2(\mathbb{Z})$ a matrix, and $h_1:H_B\to L$ a
  $\mathbb{Z}\langle t\rangle$-isomorphism such that 
  ${h_1\rtimes1}:H_B\rtimes_{\varphi_B} \langle t \rangle \rightarrow
  H_A\rtimes_{\varphi_A} \langle t \rangle$ 
  is a monomorphism with image of index~$n.$ 
\item There are $B\in SL_2(\mathbb{Z})$ a matrix and a commutative
  diagram of fiber preserving covering
  spaces 
%QQQQQQQQQQQQQQQQQQQQQQQQQQQQQQQQQQQQQQQQQQQQQQQQQQQQQQQQQQQQQQQQQQQQQQQQQQ
%QQQQQQQQQQQQQQQQQQQQQQQQQQQQQQQQQQQQQQQQQQQQQQQQQQQQQQQQQQQQQQQQQQQQQQQQQQ
%QQQQQQQQQQQQQQQQQQQQQQQQQQQQQQQQQQQQQQQQQQQQQQQQQQQQQQQQQQQQQQQQQQQQQQQQQQ
  $$
  \begin{tikzcd}
    \widetilde{M_B}\arrow{r} \arrow[swap]{d}{u} & \widetilde{M_A} \arrow{d}{v}\\
    M_B\arrow{r} & M_A\\
  \end{tikzcd}
  $$  
  where the vertical arrows, $u$ and $v$, are the infinite cyclic coverings of
  $M_B$ and $M_A$, respectively, and the horizontal arrows are
  $n$-fold coverings which are coverings of fibers.
\end{enumerate}
\end{theorem}
\begin{proof}
``(1) $\Rightarrow$ (2)''. 
Let $L\leq H_A$ be an additive subgroup.

Assume that $L$ is a $\mathbb{Z}\langle t\rangle$-submodule of $H_A$ of index
$n$. Write $H_A=\langle x, y\rangle$; then $L=\langle a_1,a_2\rangle=\langle
x^py^q,x^sy^r\rangle\leq H_A$ with 
$det
\left(
  \begin{smallmatrix}
    p & s \\
    q & r\\
  \end{smallmatrix}
\right)
=n$. Now write~$B=A|L:L\rightarrow L\in SL(2,\mathbb{Z})$, which is
the matrix $A$ 
written 
in 
terms of the basis $\{a_1,a_2\}$. 
%QQQQQQQQQQQQQQQQQQQQQQQQQQQQQQQQQQQQQQQQQQQQQQQQQQQQQQQQQQQQQQQQQQQQQQQQQQ
%QQQQQQQQQQQQQQQQQQQQQQQQQQQQQQQQQQQQQQQQQQQQQQQQQQQQQQQQQQQQQQQQQQQQQQQQQQ
%QQQQQQQQQQQQQQQQQQQQQQQQQQQQQQQQQQQQQQQQQQQQQQQQQQQQQQQQQQQQQQQQQQQQQQQQQQ

If $T_A$ is the fiber of $M_A$, let $\eta:T_B\rightarrow T_A$ be the
covering space corresponding to $L$, regarded as a subgroup of
$\pi_1(T_A)$. Let $b_1,b_2\in\pi_1(T_B)$ be the elements such that
$\eta_\#(b_i)=a_i$ ($i=1,2$). Then $H_B$ is generated by $b_1$ and
$b_2$, regarded as elements of $\pi_1(\widetilde M_B)$. We define~$h_1:H_B\rightarrow L$ as the linear extension of $b_1\mapsto a_1$ and
$b_2\mapsto a_2$. Then $h_1$ is an isomorphism, and, since $B=A|L$, it
follows that $h_1$ is a
$\mathbb{Z}\langle t\rangle$-morphism.

We have inclusions
%QQQQQQQQQQQQQQQQQQQQQQQQQQQQQQQQQQQQQQQQQQQQQQQQQQQQQQQQQQQQQQQQQQQQQQQQQQ
%QQQQQQQQQQQQQQQQQQQQQQQQQQQQQQQQQQQQQQQQQQQQQQQQQQQQQQQQQQQQQQQQQQQQQQQQQQ
%QQQQQQQQQQQQQQQQQQQQQQQQQQQQQQQQQQQQQQQQQQQQQQQQQQQQQQQQQQQQQQQQQQQQQQQQQQ
%
$$\begin{tikzcd}
H_A \arrow{r}{k} & H_A \rtimes_{\varphi_A}\langle t\rangle & \langle t\rangle
\arrow[swap]{l}{\ell},
\end{tikzcd}
$$
$$
\begin{tikzcd}
H_B \arrow{r}{i} & H_B \rtimes_{\varphi_B}\langle t\rangle & \langle t\rangle
\arrow[swap]{l}{j}.
\end{tikzcd}
$$
Write $f=(k|L)\circ h_1$, and consider the commutative diagram
$$
\begin{tikzcd}
{} & H_B \rtimes_{\varphi_B}\langle t\rangle \arrow[dashed]{ddd}{h} \\
H_B \arrow{ru}{i} \arrow[swap, bend right=10]{d}{h_1} \arrow{ddr}{f} & & \langle t\rangle \arrow[swap]{lu}{j} \arrow{ldd}{\ell}\\
L \arrow[swap, bend right=10]{rd}{k} \\
{} & H_A \rtimes_{\varphi_A}\langle t\rangle\\
\end{tikzcd}
$$
Now $\ell(t)f(m)\ell(t^{-1})=f(\varphi_A(t)(m))$. Indeed, we compute
$\ell(t)f(m)\ell(t^{-1})=(1,t)(h_1(m),1)(1,t^{-1})=(\varphi_A(t)(h_1(m)),1)$,
and also $f(\varphi_A(t)(m))= (h_1(\varphi_A(t)(m)),1)$. Since $h_1$ is a
$\mathbb{Z}\langle t\rangle$-morphism, we obtain
$\varphi_A(t)(h_1(m))=h_1(\varphi_A(t)(m))$.

Then, by the Universal Property of Semi-direct Products, the arrow
$h=f\rtimes_{\varphi_B}\ell$ such that $h(m,t)=(h_1(m),t)$ is a 
homomorphism. If $h(m,t^a)=h(n,t^b)$, then $h_1(m)=h_1(n)$ and $t^a=t^b$;
it follows that $m=n$, and $a=b$. Thus $h$ is a monomorphism.

Now the functions
$$
\begin{tikzcd}
\displaystyle
\frac{H_A\rtimes_{\varphi_A}\langle t\rangle}{L\rtimes_{\varphi_{B}}\langle t\rangle} 
\arrow[yshift=0.7ex]{r}{\alpha}
 &\displaystyle
\frac{H_A}{L} \arrow[yshift=-0.7ex]{l}{\beta}\\
\end{tikzcd}
$$
given by $$\alpha((a,t^k)\cdot L\rtimes_{\varphi_{B}}\langle
t\rangle)=\varphi_A(t^{-k})(a)\cdot L$$ and $$\beta(a\cdot
L)=(a,1)\cdot(L\rtimes_{\varphi_{B}}\langle t\rangle)$$
satisfy~$\alpha\beta=1$, and $\beta\alpha=1$.

Therefore the indices $n=[H_A:L]=[H_A\rtimes_{\varphi_{A}}\langle t\rangle:L\rtimes_{\varphi_{B}}\langle
t\rangle]$.

``(2) $\Rightarrow$ (1)''. Now assume that (2) holds. If $h_1(b)=a$, then
$(\varphi_A(t)(a),1)=(1,t)(a,1)(1,t^{-1})=h(t)h(b)h(t^{-1})=(h_1\rtimes
1)((1,t)(b,1)(1,t^{-1}))=(h_1\rtimes 1)(\varphi_B(t)(b),1)$. Thus
$\varphi_A(t)(a)=h_1(\varphi_B(h_1(b)))=\varphi_B(t)(a)$.

Then $\varphi_A(t)(L)\subset L$ and $\varphi_A(t)(a)=\varphi_B(t)(a)$ for each
$a\in L$.

As above, the functions 
$$
\begin{tikzcd}
\displaystyle
\frac{H_A\rtimes_{\varphi_A}\langle t\rangle}{L\rtimes_{\varphi_{B}}\langle t\rangle} 
\arrow[yshift=0.7ex]{r}{\alpha}
 &\displaystyle
\frac{H_A}{L} \arrow[yshift=-0.7ex]{l}{\beta}\\
\end{tikzcd}
$$
are bijections, and thus $[H_A:L]=n$.

``(2) $\Rightarrow$ (3)''. By (2), notice that $B=A|L$. Then the
matrix $B$ solves the `lifting problem' 
$$
\begin{tikzcd}
H_B\arrow{r}{B} \arrow{d} & H_B \arrow{d}\\
H_A\arrow{r}{A} & H_A\\
\end{tikzcd}
$$
where the vertical arrows are inclusions.
Thus, if $\psi:\tilde T\rightarrow T$ is the covering space of the torus
corresponding to $H_B\leq H_A$, then $B$ also solves the topological
lifting problem 
$$
\begin{tikzcd}
\tilde T\arrow{r}{B} \arrow[swap]{d}{\psi} & \tilde T \arrow{d}{\psi}\\
T\arrow{r}{A} & T\\
\end{tikzcd}
$$

One can then define $\eta:M_B\rightarrow  M_A$ an $n$-fold covering
of fibers.

Thus we obtain a diagram
$$
\begin{tikzcd}
\tilde T\times\mathbb{R}\arrow{r}{\psi\times1} \arrow[swap]{d}{u} &  T\times\mathbb{R} \arrow{d}{v}\\
M_B\arrow{r}{\eta} & M_A\\
\end{tikzcd}
$$
where the vertical arrows are infinite cyclic coverings, as required.

``(3) $\Rightarrow$ (1)''. Since we have that the diagram
$$
\begin{tikzcd}
\widetilde{M_B}\arrow{r}{\psi\times1} \arrow[swap]{d}{u} & \widetilde{M_A} \arrow{d}{v}\\
M_B\arrow{r} & M_A\\
\end{tikzcd}
$$  
commutes, 
we see that $\psi\times1$ is compatible with the covering transformations of
$\widetilde{M_B}$ and $\widetilde{M_A}$; that is,
$(\psi\times1)\circ\varphi_B=\varphi_A\circ(\psi\times1)$. It
follows that  $H_B$
is a $\mathbb{Z}\langle t\rangle$-submodule of $H_A$.
\end{proof}

\begin{corollary}
\label{coro41}
	Let $M_A$ be a torus bundle and	let $\varphi:\tilde{M}\to M_A$ be an
        $n$-fold covering of fibers.
If~$A=\left(
	\begin{smallmatrix}
	\alpha & \beta \\
	\gamma & \delta
	\end{smallmatrix}
	\right),
	$
then $\tilde{M}= M_{B},$ where 
$$
B=
\left(
  \begin{array}{cc}
    \displaystyle
    \frac{det\left(
        \begin{smallmatrix}
          p\alpha + q\beta & s\\
          p\gamma + q\delta & r
        \end{smallmatrix}
      \right)}{n}
    &
    \displaystyle
    \frac{det\left(
     	\begin{smallmatrix}
          s\alpha + r\beta & s\\
          s\gamma + r\delta & r
     	\end{smallmatrix}
      \right)}{n}
    \\  
    & \\
    \displaystyle
    \frac{det\left(
        \begin{smallmatrix}
          p & p\alpha + q\beta\\
          q & p\gamma + q\delta
        \end{smallmatrix}
      \right)}{n}
    &
    \displaystyle
    \frac{det\left(
        \begin{smallmatrix}
          p & s\alpha + r\beta \\
          q & s\gamma + r\delta
        \end{smallmatrix}
      \right)}{n}
  \end{array}
\right),
$$	
and
$\left(
  \begin{smallmatrix}
    p & s \\
    q & r\\
  \end{smallmatrix}
\right)
$
has determinant $n$, and corresponds to the subgroup~$H_B=\langle
x^py^q,x^sy^r\rangle\leq H_A$ determined by $\varphi$.		 
\end{corollary}
\begin{proof}
 Write $a_1=x^py^q$, and $a_2=x^sy^r$. 
As in the proof of ``(1)$\Rightarrow$(2)'' in the
theorem, $B$ is the matrix $A$ written in terms of the basis $a_1,a_2$.
If
$B=\left(
  \begin{smallmatrix}
    a & b \\
    c & d\\
  \end{smallmatrix}
\right)
$, then $Aa_1=a_1^aa_2^c$, $Aa_2=a_1^ba_2^d$ translates into the
linear systems
$$
\begin{array}{rlrl}
pa+sc&=\ p\alpha+q\beta &\quad pb+sd&=\ s\alpha+r\beta\\
qa+rc&=\ p\gamma+q\delta &\quad qb+rd&=\ s\gamma+r\delta.\\
\end{array}
$$
Solving for $a,b,c,d$ gives $B$ the form of the statement.

\end{proof}

%%%%%%%%%%%%%%%%%%%%%%%%%%%%%%%%%%%%%%%%%%%%%%%%%%%%%%%%%%%%%%%%%%%%%%%
\begin{remark} 
\label{rem43}
For an integral matrix
$\left(
  \begin{smallmatrix}
    p & s \\
    q & r\\
  \end{smallmatrix}
\right)
$
with determinant~$n\geq2$, we have a function\ \ 
$\widetilde{\ \ }:GL(2,\mathbb{Z})\rightarrow GL(2,\mathbb{Z})$ such that
$A\mapsto \widetilde A$, where, 
if~$A=\left(
  \begin{smallmatrix}
    \alpha & \beta \\
    \gamma & \delta
  \end{smallmatrix}
\right)
$, 
then $\widetilde A$ is as the matrix $B$ in
Corollary~\ref{coro41}. Also if~$A\in SL(2,\mathbb{Z})$, then
$\widetilde A\in SL(2,\mathbb{Z})$. This function is multiplicative:~$\widetilde{A_1\cdot A_2}={\widetilde{A_1}}\cdot{\widetilde{A_2}}$, and
preserves inverses: $\widetilde{A^{-1}}={\widetilde{A}}^{-1}$. That is,
it is a group homomorphism. We do
not know if this last property make any geometrical sense in terms of
the covering spaces involved. 
\end{remark}

\begin{lemma}
\label{lemma44}
	Let $X,$ $Y$ and $Z$ path-connected topological spaces. 
If in the following pushout commutative diagram
$$
\begin{tikzcd}
X\arrow{r}{g} \arrow[swap]{d}{\varphi} & Z \arrow{d}{\psi}\\
Y\arrow[swap]{r}{h} & P\\
\end{tikzcd}
$$
the arrow $\varphi: X\to Y$ is a covering space and  the arrow $g:X\to Z$ is a
homeomorphism, then~$\psi: Z\to P$ is a covering space and $h:Y\to
P$  is a homeomorphism.

In particular $\varphi$ and $\psi$ have the same number of sheets.
\end{lemma}
\begin{proof}
We may assume that
$$P=\displaystyle\frac{Y\sqcup Z}{g(x)\sim \varphi(x), \forall x\in X}$$
and that $\psi$ and $h$ are the corresponding inclusions
followed by the quotient~$\pi:Y\sqcup Z\rightarrow P$.

Note that, if~$w\in P$, then~$w=\{a\}\cup\{g(x):x\in \varphi^{-1}(a)\}$ for
some~$a\in Y$. Then $h$ and $\psi$ clearly are surjective, and $h$ is
one-to-one. 

For~$U\subset P$, if~$h^{-1}(U)$ is open, since the square above commutes and~$g$ is a 
homeomorphism, it follows that $\psi^{-1}(U)$ also is open; 
therefore $\pi^{-1}(U)\cap Y$ and $\pi^{-1}(U)\cap Z$ are open in $Y$ and $Z$,
respectively. It follows that $\pi^{-1}(U)$ is open in~$Y\sqcup Z$, and,
therefore,~$U$ is open in $P$. Thus $h$ is a homeomorphism. 

Now for
$W\subset P$, if
$\psi^{-1}(W)$ is open, since~$\varphi$ is 
an identification, it follows that~$h^{-1}(W)$ is open, and thus $W$ is open
in $P$. That is,~$\psi$ is an
identification. 

If $w\in P$, 
then there is a fundamental neighborhood $V\subset Y$ of~$h^{-1}(p)$ for
the covering $\varphi$. Then~$h(V)$ is a fundamental neighborhood of
$p$ for~$\psi$. Thus $\psi$ is a covering space.

\end{proof}

\begin{proposition}
\label{prop45}
Let $\varphi:M_B\rightarrow M_A$ be an $n$-fold  covering of
fibers of torus bundles, where $B=\widetilde A$ as in Remark~\ref{rem43}.

For any matrix $D$ such that $B$ is conjugate to $D$ in $GL(2,\mathbb{Z})$, there exists an
$n$-fold covering  
of fibers $\psi:M_D\rightarrow M_C$ with $D=\widetilde C$ as in
Remark~\ref{rem43}, and $M_C\cong M_A$. 
\end{proposition}
\begin{proof}
Assume that $D=gBg^{-1}$. Write $T_B$ for the fiber of $M_B$. Then we have
a commutative diagram
$$
\begin{tikzcd}
T_D\arrow{r}{D} \arrow[swap]{d}{g} & T_D \arrow{d}{g}\\
T_B\arrow[swap]{r}{B} & T_B\\
\end{tikzcd}
$$
where $T_D$ is a torus ($T_D=g^{-1}(T_B)$). This gives a fiber preserving homeomorphism
$g:M_B\rightarrow M_D$. Taking the pushout of $\varphi$ and $g$, we
obtain $h:M_A\rightarrow M$ a homeomorphism and $\psi:M_D\rightarrow M$
a covering space with $\psi g=h\varphi$ as in
Lemma~\ref{lemma44}. Write  $T_A$ for the fiber of
$M_A$,~$T_C=h(T_A)$, and $C=hAh^{-1}:T_C\rightarrow T_C$. Then
$M=M_C$. 
Since $\psi D=\psi gBg^{-1}=h\varphi Bg^{-1}=hA\varphi g^{-1}=C\psi$,
we have that $\psi$ is a covering of fibers by Theorem~\ref{lemma42},
and $D=\widetilde C$, 
with $p,q,s,r$ given 
by the subgroup $H_D\leq H_C$.

\end{proof}

\subsection{Cyclic coverings of the torus}
\label{sec41}
An $n$-fold covering space $\eta:X\rightarrow Y$ is called
\emph{cyclic} if the associated representation
$\omega_\eta:\pi_1(Y)\rightarrow S_n$ has image $\omega(\pi_1(Y))\cong
\mathbb{Z}_n$, a
cyclic group. Write $\varepsilon_n=(1,2,\dots,n)\in S_n$ for the standard
$n$-cycle.

For a torus $T$ with $\pi_1(T)=\langle
a,b:[a,b]\rangle$, a cyclic covering~$\eta:\tilde T\rightarrow T$ with $\tilde
T$ connected,
has associated representation $\zeta_\eta:\pi_1(T)\rightarrow S_n$
such that $a\mapsto \varepsilon_n^\sigma$, $b\mapsto
\varepsilon_n^\tau$ with, say, $(n,\sigma)=1$. Now $\zeta_\eta$ is conjugate
to $\omega_\rho:\pi_1(T)\rightarrow S_n$
such that $a\mapsto \varepsilon_n$, $b\mapsto
\varepsilon_n^\rho$ for some integer $\rho$. The covering space
equivalence class of $\eta$, has a unique representative~$\eta_\rho$
with associated representation $\omega_\rho$ (for uniqueness we
assume that, say, $\rho\in\{0,\dots,n-1\}$).

In the one-to-one correspondence between coverings of $T$ and
subgroups of $\pi_1(T)$,  we have that
$\eta_\rho$ corresponds to 
the subgroup
$(\eta_\rho)_\#(\pi_1(\tilde T))=\langle a^n,a^{-\rho}b\rangle$
(see~\cite{trefoil}, Lemma in p.5).

\subsection{Non-cyclic coverings of the torus}
For positive integers $m,n,d$, and $i_0$ such that $m$ divides $n$,
$dm$ divides $n$, and
$(d,i_0)=1$ with $0\leq i_0\leq d-1$, write
$\rho=i_0n/d$. %$\rho=i_0m$. 
%\textbf{!`\'Orale!}
%%%%%%%%%%%%%%%%%%%%%%%%%%%%%%%%%%%%%%%%%%%%%%%%%%%%%%%%%%%%%
%%%%%%%%%%%%%%%%%%%%%%%%%%%%%%%%%%%%%%%%%%%%%%%%%%%%%%%%%%%%%
%%%%%%%%%%%%%%%%%%%%%%%%%%%%%%%%%%%%%%%%%%%%%%%%%%%%%%%%%%%%%
%%%%%%%%%%%%%%%%%%%%%%%%%%%%%%%%%%%%%%%%%%%%%%%%%%%%%%%%%%%%%
%
We construct
a 
transitive representation $\omega(m,n,d,\rho)$ of
$\pi_1(T)=\langle a,b:[a,b]\rangle$ into~$S_{mn}$ with image isomorphic to
$\mathbb{Z}_m\oplus\mathbb{Z}_n$. 

For $j=0,\dots,m-1$, define
\begin{equation}
\label{eq1}
\sigma_{j+1}=(jn+1,jn+2,\dots,jn+n).
\end{equation}
Then $\sigma=\sigma_1\cdots\sigma_m$ is a product of $m$ cycles of order $n$.

If $d=1$, then $i_0=0$, and $\rho=0$. For $j=1,\dots,n$, define
$$\tau_{j}=(j,n+j,2n+j,\dots,(m-1)n+j).$$
Then $\tau=\tau_1\cdots\tau_n$ is a
product of $n$ cycles of order $m$, and
$\omega:\pi_1(T)
%\mathbb{Z}_m\oplus\mathbb{Z}_n
\rightarrow S_{mn}$ given by 
$a\mapsto \sigma$, and $b\mapsto\tau$, is a
representation~$\omega(m,n,1,0)$ as required.

If $d>1$, then $i_0>0$, and $\rho>0$. Define
$$\tau_{k,j}=r_k+(j-1)n$$
where $r_k$ is the number $(k-1)\rho+1$ reduced $\pmod n$, and $1\leq j\leq
m$, and $1\leq k\leq d$. 
Then
$\tau_1=(\tau_{1,1},\tau_{1,2},\dots,\tau_{d,m})$ is a cycle of
order~$dm$.

Now, for $1\leq\ell\leq n/d-1$,
define
$$\tau_{\ell+1}=(\tau_{1,1}+\ell,\tau_{1,2}+\ell,\dots,\tau_{d,m}+\ell).$$
Then
$\tau=\tau_1\tau_2\cdots\tau_{n/d}$
is a product of $n/d$ cycles of order $dm$, and
$\omega:\pi_1(T)
%\mathbb{Z}_m\oplus\mathbb{Z}_n
\rightarrow S_{mn}$ given by 
$a\mapsto \sigma$, and $b\mapsto\tau$, is a
representation~$\omega(m,n,d,\rho)$ as required. 

For example, for $m=2,n=8,d=4,i_0=1$, and $\rho=2$, %delta=3,
the representation
$\omega(2,8,4,2):\pi_1(T)
%\mathbb{Z}_2\oplus\mathbb{Z}_8
\rightarrow S_{16}$ is 
given by
$$a\mapsto (1,2,3,4,5,6,7,8)(9,10,11,12,13,14,15,16)$$
$$b\mapsto  (1,9,3,11,5,13,7,15)(2,10,4,12,6,14,8,16).$$
%[ (1,2,3,4,5,6,7,8)(9,10,11,12,13,14,15,16),
%  (1,9,5,13)(2,10,6,14)(3,11,7,15)(4,12,8,16) ], # m=2,n=8,d=2,rho=4,delta=5,

\begin{lemma}
\label{lemma46p}
If $\omega:\pi_1(T)
\rightarrow S_{k}$ is a
%faithful 
transitive representation with image isomorphic to
$\mathbb{Z}_m\oplus\mathbb{Z}_n$, and with $m$ a submultiple of $n$, then
there are non-negative integers $d$ and~$\rho$ such that $\omega$ is
conjugate to $\omega(m,n,d,\rho)$.
\end{lemma}
\begin{proof}
Since $\omega$ is 
%faithful and 
transitive, then $\omega$ is regular
and $k=mn$. Since $m$ divides $n$, we see that the order of, say,
$\omega(a)$ is $n$, and that $m$ divides the order of~$\omega(b)$.
We obtain
$order(\omega(b))=dm$ for some $d$. Also, since~$x^n=1$ for al
$x\in Image(\omega)$, then $dm$ divides~$n$. 

For $j=0,\dots,m-1$ write $O_{j+1}$ for the orbit of $\omega(a)$ that
contains~$\omega(b)^j(1)$, and $\gamma_{j+1}$ for the corresponding
$n$-cycle of $\omega(a)$ which acts on $O_{j+1}$ in the disjoint cycle
decomposition $\omega(a)=\gamma_1\cdots\gamma_m$.  For $j=1,\dots,m$, take
$v\in S_{mn}$ 
such that 
$v(\omega(b)^j(1))=1+jn$, and
$v\gamma_jv^{-1}=\sigma_j$ with $\sigma_j$ as 
in Equation~\ref{eq1} above; this is possible for, the orbits of the
$\gamma_j$'s are disjoint. Then, if
$\rho+1=(v\cdot\omega(b)\cdot v^{-1})^{m-1}(1)$, we have that $\omega$ is
conjugate to~$\omega(m,n,d,\rho)$; that is, $v\cdot\omega\cdot
v^{-1}=\omega(m,n,d,\rho)$. 

\end{proof}

\begin{remark}
Notice that the $n$-fold cyclic covering of the torus $\eta_\rho$ as in
Section~\ref{sec41}, can be regarded as the associated covering of
$\omega(1,n,d,\rho)$ with 
$d=order(\varepsilon_n^\rho)=n/(n,\rho)$. 
\end{remark}

\begin{lemma}
\label{lemma48}
Let~$\eta:\tilde T\rightarrow T$ be an 
$mn$-fold non-cyclic covering space of the torus, $T$, with $m$ a submultiple of
$n$ and $\tilde T$ connected. Let~$\omega:\pi_1(T)\rightarrow S_{mn}$ be the
representation associated to~$\eta$. Then, as in
Lemma~\ref{lemma46p}, $\omega$ is 
conjugate to~$\omega(m,n,d,\rho)$ for some integers $d$ and~$\rho$.

If, say, 
$\omega(a)$ has order $n$, then for
any integer $r$ such that $r\equiv\rho\mod n$, there is a basis
$\tilde a$, $\tilde b$ of $\pi_1(\tilde T)$ such that $\eta_\#(\tilde
a)= a^n$, and $\eta_\#(\tilde b)= a^{-r}b^m$.
\end{lemma}
\begin{proof}
The proof goes as the proof in~\cite{trefoil},  Lemma in p.~5.
\end{proof}

\begin{corollary}
\label{coro48}
Let $\eta:\tilde T\rightarrow T$ be an 
$mn$-fold covering space with $m$ a submultiple of $n$ such that the
image of its associated representation is isomorphic
to~$\mathbb{Z}_m\oplus\mathbb{Z}_n$ (we  
allow $m=1$).

There is an integer $\rho$ such that $m$ divides $\rho$,
and such that 
the subgroup of 
$\pi_1(T)$ corresponding to $\eta$ is $\eta_\#(\pi_1(\tilde T))=\langle
a^n,a^{-\rho}b^m\rangle$. 
\end{corollary}

\subsection{Coverings of torus bundles that lower the genus}
Let $\eta:\tilde T\rightarrow T$ be a finite-sheeted covering space of the
torus $T$ with $\tilde T$ connected. We say that $\eta$ extends to a covering
of torus bundles if there is a covering of fibers between torus bundles,
$\varphi:\tilde M\rightarrow M$, such that $\varphi$ restricted to the fiber
of $\tilde M$ equals $\eta$.

\begin{lemma}
\label{lemma46}
Let $\eta:\tilde T\rightarrow T$ be a covering space of the
torus $T$ with associated representation $\omega(m,n,d,\rho)$, $m\geq1$. Write 
$\rho=i_0n/d$.

Then $\eta$ extends to a covering of torus bundles $\tilde M\rightarrow M_A$
if and only if there are integers $p,k,r$, and $s$ such that
{\arraycolsep=7.8pt\def\arraystretch{1.4}
$$A=
\left(
\begin{array}{cc}

p-i_0\frac{n}{dm}k 
& \frac{n}{m} r-i_0\frac{n}{dm} s+i_0\frac{n}{dm}(p-i_0\frac{n}{dm} k) \\
%\displaystyle
 k & s+i_0\frac{n}{dm}k\\
\end{array}
\right)
$$
}
%In particular, when $p-i_0\frac{n}{dm} k=-1$,
%
%$$A=\left(
%\begin{matrix}
%-1 & k\\
%\gamma & -k\gamma-1\\
%\end{matrix}
%\right)
%$$
%
%where $\gamma=\frac{n}{m}r-i_0\frac{n}{dm} s-i_0\frac{n}{dm}$.

\end{lemma}
\begin{proof}
Recall that if
$A=\left(
\begin{smallmatrix}
\alpha & \beta\\
\gamma & \delta\\
\end{smallmatrix}
\right)$, then $A$ acts in $\pi_1(T)$ as $Ax=x^\alpha y^\gamma$
and~$Ay=x^\beta y^\delta$. 

Write $a_1=x^n$, and $a_2=x^{-\rho}y^m$; then, by Corollary~\ref{coro48},
$L=\langle a_1,a_2\rangle\leq\pi_1(T)$ is the subgroup corresponding to
$\eta$.
By Theorem~\ref{lemma42}, $\eta$ extends to a covering of torus bundles $\tilde
M\rightarrow M_A$ if and only if $AL\subset L$.

If $A$ has the form of the statement of the lemma, then $Aa_1=a_1^pa_2^{kn/m}$,
 $Aa_2=a_1^ra_2^s\in L$, and we conclude that $\eta$ extends to a covering of torus
bundles $\tilde 
M\rightarrow M_A$.

If $AL\subset L$, then $Aa_1=a_1^pa_2^q$ and $Aa_2=a_1^ra_2^s$
for some integers $p,q,r,s$. 
If we write
$A=\left(
\begin{smallmatrix}
\alpha & \beta\\
\gamma & \delta\\
\end{smallmatrix}
\right)$, then
last equations are equivalent to
$$n\alpha=np-\rho q, \quad  -\rho\alpha+m\beta=nr-\rho s $$
$$n\gamma=qm, \quad -\rho\gamma+m\delta=ms.$$
We see that $q$ is of the form $q=\frac{n}{m}k$, and the lemma follows, that
is, 
$$\textstyle\alpha=p-i_0 \frac{n}{dm}k, \quad  \textstyle\beta=\frac{n}{m}r-i_0\frac{n}{dm} s+i_0\frac{n}{dm}(p-i_0\frac{n}{dm}k) $$
$$\gamma=k, \quad \delta=s+i_0\frac{n}{dm} k.$$
%

%In the special case $p-i_0\frac{n}{dm} k=-1$, we have $\alpha=-1$, $\beta=k$,
%and, since $1=\alpha\delta-\beta\gamma=-\delta-k\gamma$, we obtain $\delta=-k\gamma-1$.

\end{proof}

\begin{remark} 
If $A$ is as in the statement of Lemma~\ref{lemma46}, then
$$
\left(
\begin{array}{cc}
 1 & i_0\frac{n}{dm} \\
 0 & 1
\end{array}
\right)
A
\left(
\begin{array}{cc}
 1 & i_0\frac{n}{dm} \\
 0 & 1
\end{array}
\right)^{-1}=
\left(
\begin{array}{cc}
 p & r\frac{n}{m} \\
 k & s
\end{array}
\right).
$$
That is, it is rather common for a torus bundle $M_A$ to admit a
covering of fibers.

\end{remark}
	
\begin{theorem}
\label{thm46}
Let $\varphi:\tilde M\rightarrow M$ an $mn$-fold covering of torus
bundles which is a covering of fibers with $m$ a divisor of $n$. 

The 
genus $g(\tilde M)<g(M)$ if and only if $m<n$, and~$M\cong M_A$,
and~$\tilde M\cong M_B$, where  
$A=\left(
	\begin{array}{cc}
	-1 & -\frac{n}{m} \\
	 a & a\frac{n}{m}-1
	\end{array}
	\right)
	$
and 
$
	B=\left(
	\begin{array}{cc}
	-1 & -1 \\
	a\frac{n}{m} & a\frac{n}{m}-1
	\end{array}
	\right)
	$ for some integer $a\neq\pm1$.

Moreover if $A$ has the form above, then the covering space of the torus
associated to the representation $\omega(n,m,d,\rho)$ with $\rho=i_0n/d=i_0m\ell$,
extends to an ${mn}$-fold covering of fibers
of $M_B$ onto $M_A$.
% if and only if 
%${n}/{m}$ divides the number~$a((i_0\ell)^2+1)^2$.
\end{theorem}

\begin{proof}
If matrices $A$ and $B$ have the form of the statement, then
$2=g(M_B)<g(M_A)=3$. See Section~\ref{sec22}. 

Assume that $g(\tilde M)<g(M)$. Then $g(\tilde M)=2$, and $g(M)=3$. By
Proposition~\ref{prop45}, we may assume that $\tilde M=M_B$ with 
$B=
\left(
\begin{smallmatrix}
 -1 & -1 \\
 b & b-1
\end{smallmatrix}
\right)
$
for some in\-te\-ger~$b$,
and $M=M_A$ where
$A$ is some matrix
$A=
\left(
\begin{smallmatrix}
 \alpha & \beta \\
 \gamma & \delta
\end{smallmatrix}
\right)
$, and~$B=\widetilde A$.

The submodule of the infinite cyclic covering of $M_A$ corresponding to the
covering $\varphi$ is $H_B=\langle x^n,x^{-\rho}y^m  \rangle$ with $m$ a
divisor of $n$, and
$\rho=i_0m\ell$ ($=i_0n/d$). See 
Corollary~\ref{coro48} (and Lemma~\ref{lemma48}). 

As in Remark~\ref{rem43}, with $p=n$, $q=0$, $r=m$, and $s=-i_0m\ell$, 
$$
\widetilde A=
\left(
\arraycolsep=1.8pt\def\arraystretch{2.2}
\begin{array}{cc}
 \alpha +\gamma{i_0}\ell & \null\quad
\displaystyle\left(-\gamma  {i_0}^2\ell^2-\alpha 
   {i_0}\ell+\delta  {i_0}\ell+\beta \right) \frac{m}{n} \\
 \displaystyle\frac{\gamma  n}{m} & \delta -\gamma  {i_0}\ell\\
\end{array}
\right)
$$
Since we are assuming
$\widetilde A=
\left(
\begin{smallmatrix}
 -1 & -1 \\
 b & b-1
\end{smallmatrix}
\right),
$
we see that
$$
\arraycolsep=1.8pt\def\arraystretch{2.2}
\begin{array}{ll}
\displaystyle
\alpha = -\frac{b {i_0}\ell m+n}{n},&
\displaystyle
\beta= -\frac{b {i_0}^2\ell^2 m^2+b {i_0}\ell mn+n^2}{m n}\\
\displaystyle
\gamma = \frac{bm}{n},&
\displaystyle
\delta= \frac{b {i_0}\ell m+b n-n}{n}\\
\\
\end{array}
$$
Since $\gamma$ is an integer, it follows that $n$ divides $bm$, say, $b=an/m$. Then
$$
\begin{array}{ll}
%\displaystyle
\alpha = -a{i_0}\ell -1,&
%\displaystyle
\beta = -a{i_0}^2\ell^2 -a{i_0}\ell \frac{n}{m} -\frac{n}{m}\\
%\displaystyle
\gamma = a,&
%\displaystyle
\delta = a{i_0}\ell +a \frac{n}{m}-1\\
\\
\end{array}
$$
and
$$
A=
\left(
\begin{array}{cc}
-a {i_0}\ell-1 & -a {i_0}^2\ell^2-a \frac{n}{m} {i_0}\ell-\frac{n}{m}\\
a & a {i_0}\ell+a \frac{n}{m}-1\\
\end{array}
\right)
$$

Notice that 
$$
\left(
\begin{array}{cc}
 1 & i_0\ell \\
 0 & 1
\end{array}
\right)
A
\left(
\begin{array}{cc}
 1 & i_0\ell \\
 0 & 1
\end{array}
\right)^{-1}=
\left(
\begin{array}{cc}
 -1 & -\frac{n}{m} \\
 a & a \frac{n}{m}-1
\end{array}
\right)
$$
Since in the conjugacy class of $A$ in $GL(2,\mathbb{Z})$, except for
interchange of $a$ and $n/m$, there is no other matrix in the %Sakuma's
form of Section~\ref{sec22},
we see that $a\neq\pm1$ and $m<n$, for we are assuming $g(M_A)=3$. 

By Lemma~\ref{lemma46}, the covering of the torus associated to the
representation $\omega(m,n,d,i_0n/d)$ with $n/d=m\ell$, extends to a
covering of fibers $\tilde M\rightarrow M_A$ if and only if there are integers
$p,k,r,s$ such that
$$
A=
\left(
\begin{array}{cc}
p-i_0\ell k & \frac{n}{m} r-i_0\ell s+i_0\ell(p-i_0\ell k) \\
k & s+i_0\ell k\\
\end{array}
\right).
$$
Defining $k=a$, $p=-1$, $s=a\frac{n}{m}-1$, and $r=-1$, we obtain the
required equality. 
And the theorem
follows.

%%%%%%%%%%%%%%%%%%%%%%%%%%%%%%%%%%%%%%%%%%%%%%%%%%%%%%%%%%%%%%%%%%%%%
%%%%%%%%%%%%%%%%%%%%%%%%%%%%%%%%%%%%%%%%%%%%%%%%%%%%%%%%%%%%%%%%%%%%%
%%%%%%%%%%%%%%%%%%%%%%%%%%%%%%%%%%%%%%%%%%%%%%%%%%%%%%%%%%%%%%%%%%%%%
%%%%%%%%%%%%%%%%%%%%%%%%%%%%%%%%%%%%%%%%%%%%%%%%%%%%%%%%%%%%%%%%%%%%%
%%%%%%%%%%%%%%%%%%%%%%%%%%%%%%%%%%%%%%%%%%%%%%%%%%%%%%%%%%%%%%%%%%%%%
\end{proof}

\begin{remark}
A  representation
$\omega:\pi_1(T)\rightarrow S_{n^2}$ with image
$\mathbb{Z}_n\oplus\mathbb{Z}_n$ is
conjugate to $\omega(n,n,1,0)$.  The subgroup of the corresponding
covering space is $L=\langle x^n,y^n\rangle$. If 
$A\in SL(2,\mathbb{Z})$, then $AL\subset L$, and $\widetilde
A=A$. Then the extension to a $n^2$-fold covering of fibers is of the form
$M_A\rightarrow M_A$, and there is no genus lowering.
\end{remark}

\begin{remark}
Theorem~\ref{thm46} implies that,
if
$A=
\left(
\begin{smallmatrix}
 -1 & -a \\
 b & ab-1
\end{smallmatrix}
\right)
$
with $a>0$ and $|a|,|b|\neq1$, then for each positive integer $m$, the
torus bundle $M_A$ admits an $(am)$-fold covering space that lowers the
genus.
\end{remark}

\section{Seifert manifolds}
\label{seifert}

Let $M$ be the orientable Seifert manifold with orientable orbit
surface of genus~$g$ and Seifert symbol $(Oo,
g;\beta_1/\alpha_1,\dots,\beta_t/\alpha_t)$, where
$\alpha_i,\beta_i$ are integers with
$\alpha_i\geq 1$ and $(\alpha_i,\beta_i)=1$ for $i=1,\dots,t$. 

Then
the fundamental group 
$\pi_1(M)=\langle
a_1,b_1,\dots,a_g,b_g,q_1,\dots,q_t,h:q_1^{\alpha_1}h^{\beta_1}=1,\dots,
q_t^{\alpha_t}h^{\beta_t}=1,q_1\cdots q_t=[a_1,b_1]\cdots[a_g,b_g],
[h,q_i]=[h,a_j]=[h,b_j]=1\rangle$ where $a_1,b_1,\dots,a_g,b_g$
represent a basis for the fundamental group of the orbit surface
of~$M$. By Lemma~1 of~\cite{trefoil} one obtains 

\begin{lemma}
\label{lema1}
Let $M=(Oo,
g;\beta_1/\alpha_1,\dots,\beta_t/\alpha_t)$ be a Seifert manifold.
Let $r_1,\dots,r_t$ be integers such that $\alpha_ir_i+\beta_i\equiv 0
\mod{n}$ for $i=1,\dots,t$, and assume that $r_1+\cdots+r_t=0$. Then there is an
$n$-fold cyclic covering space 
$$(Oo,g;B_1/\alpha_1,\dots,B_t/\alpha_t)\rightarrow M$$
where the integer
$B_i=({\alpha_ir_i+\beta_i})/{n}$
for $i=1,\dots,t$.
\end{lemma}

\begin{lemma}
\label{lema23}
Let $M$ be the Seifert manifold with symbol $(Oo,g;\beta/\alpha)$ and
$g\geq0$. Then 
the Heegaard genus of $M$ is
$$
h(M)=\left\{ 
  \begin{array}{ll}
    2g & \textrm{if $\beta=\pm1$}\\
    2g+1 & \textrm{otherwise.}\\
  \end{array} 
\right.
$$

\end{lemma}

\begin{proof}
One can construct a Heegaard
decomposition for $M$ of genus~$2g$ if $\beta=\pm1$, and a Heegaard
decomposition for $M$ of genus $2g+1$ if~$\beta\neq\pm1$ (see~\cite{walo}). Therefore
$
h(M) \leq \left\{ 
  \begin{array}{ll}
    2g & \textrm{if $\beta=\pm1$}\\
    2g+1 & \textrm{otherwise.}\\
  \end{array} 
\right.
$

Recall that $rank(H_1(M))\leq h(M)$. 

Since 
$H_1(M)=\langle a_1,b_1,\dots,a_g,b_g,q,h:q^\alpha
h^\beta=1, q=1\rangle_{Ab}$,
%\cong\mathbb{Z}^{2g}\oplus\langle h: h^\beta=1\rangle$,
%
then
$H_1(M)\cong
\left\{ 
  \begin{array}{ll}
    \mathbb{Z}^{2g} & \textrm{if $\beta=\pm1$}\\
    \mathbb{Z}^{2g}\oplus \mathbb{Z}_{|\beta|} & \textrm{otherwise}\\
  \end{array} 
\right.
$
 where the subindex `$Ab$' indicates the image of the Abelianization homomorphism.
In particular 
$h(M)\geq \left\{ 
  \begin{array}{ll}
    2g & \textrm{if $\beta=\pm1$}\\
    2g+1 & \textrm{otherwise.}\\
  \end{array} \right.$

\end{proof}

\begin{corollary}
\label{coro22}
For any integers $g\geq0$, $\alpha\geq 1$, and $|\beta|\geq2$ with
$\alpha$ and $\beta$ coprime, there is a $|\beta|$-fold
covering space 
$$(Oo,g;\pm1/\alpha)\rightarrow (Oo,g;\beta/\alpha).$$
And the genus $g(Oo,g;\beta/\alpha)=g(Oo,g;\pm1/\alpha)+1$.
\end{corollary}

\begin{proof}
If we set $r_1=0$, then, using Lemma~\ref{lema1}, we obtain
$B_1=\beta/|\beta|=\pm1$, and a $|\beta|$-fold covering space
$(Oo,g;\pm1/\alpha)\rightarrow (Oo,g;\beta/\alpha)$. 
\end{proof}

\begin{remark}
In the case $M$ is an orientable Seifert manifold with non-orientable
orbit surface, the following also holds.

\begin{theorem}[\cite{jair}]
\label{prop55}
Let $\alpha, \beta$ be a pair of coprime integers with
$\alpha\geq1$ and~$|\beta|\geq2$; let $g<0$, and
let $M$ be the Seifert manifold with symbol $(Oo,g;\beta/\alpha)$. 

If $g<-1$, then $\pi_1(M)$ is of infinite order and $M$ has a finite
covering space $\tilde M=(Oo,g;\pm1/\alpha)\rightarrow M$ such that
the Heegaard genus 
$h(\tilde M)=h(M)+1$. Also $rank(\pi_1(\tilde M))=rank(\pi_1(M))+1$.
\end{theorem}

Also it follows from~\cite{jair}, that the manifolds of Corollary~\ref{coro22}
and Theorem~\ref{prop55} are the only examples of
 (branched or unbranched) coverings of (orientable or not) Seifert manifolds
 that lower the 
Heegaard genus, in case 
the orbit surface is not the 2-sphere.

\end{remark}


\begin{thebibliography}{99}

\bibitem{walo}
M. Boileau and H. Zieschang.
Heegaard genus of closed orientable Seifert 3-manifolds.
Invent. Math. 76 (1984), no. 3, 455--468. 

\bibitem{toreador}
T. Li.
Rank and genus of 3-manifolds.
J. Amer. Math. Soc. 26 (2013), no. 3, 777--829. 

\bibitem{Lu} M.~Lustig.
On the rank, the deficiency and the homological dimension of groups:
the computation of a lower bound via Fox ideals. Lecture Notes in Math. 1440,
164--174, Springer, Berlin, 1990. 

\bibitem{trefoil}
V. N\'u\~nez and E. Ram\'\i rez.
The trefoil knot is as universal as it can be.
Topology Appl. 130 (2003), no. 1, 1--17. 

\bibitem{jair}
J.~Remigio. Thesis. Cimat, M\'exico, 2008.

\bibitem{saku1}
M.~Sakuma.
Surface bundles over $S^1$ which are 2-fold branched cyclic coverings of $S^3$.
Math. Sem. Notes Kobe Univ. 9 (1981), no. 1, 159--180. 

\bibitem{chultens}
J. Schultens and R. Weidman, Richard.
On the geometric and the algebraic rank of graph manifolds.
Pacific J. Math. 231 (2007), no. 2, 481--510. 

\bibitem{chalen}
P. Shalen.
Hyperbolic volume, Heegaard genus and ranks of groups. Workshop on
Heegaard Splittings, 335--349, 
Geom. Topol. Monogr., 12, Geom. Topol. Publ., Coventry, 2007. 

\end{thebibliography}
\end{document}